\documentclass[a4paper,12pt]{article}
\usepackage{array,amsfonts,amsmath,graphicx,mathrsfs,amssymb,amsthm,latexsym}
\usepackage[latin1]{inputenc}
\usepackage{color}
\def \msp {\vspace{-1ex}}

\newtheorem{thm}{Theorem}[section]
\newtheorem{lemma}[thm]{Lemma}

\def \cF {{\cal F}}
\def \cG {{\cal G}}
\def \cH {{\cal H}}

\def \cP {{\cal P}}

\begin{document}
\title{\vspace{-10ex} ~~ \\
Resolvable  3-star designs}
\author {Selda K\"{u}\c{c}\"{u}k\c{c}\.{i}f\c{c}\.{i}
\thanks{Research supported by Scientific and Technological Research Council of Turkey Grant Number: 110T692}\\
\small \msp Department of Mathematics \\
\small \msp Ko\c{c} University \\
\small \msp Istanbul\\
\small Turkey\\
{\small \vspace{1ex} \tt skucukcifci@ku.edu.tr} \\  
Giovanni Lo Faro
\thanks{Corresponding author}
\thanks{Supported  by PRIN, PRA and I.N.D.A.M (G.N.S.A.G.A.), Italy}\\
\small \msp Dipartimento di Matematica e Informatica \\
\small \msp Universit\`a di Messina \\
\small \msp Messina\\
\small Italia\\
{\small \vspace{1ex} \tt lofaro@unime.it}\\
 Salvatore Milici
\thanks{Supported  by I.N.D.A.M (G.N.S.A.G.A.), Italy}\\
\small \msp Dipartimento di Matematica e Informatica \\
\small \msp Universit\`a di Catania \\
\small \msp Catania\\
\small Italia\\
{\small \vspace{1ex} \tt milici@dmi.unict.it}\\  
Antoinette Tripodi
\thanks{Supported  by PRIN, PRA and I.N.D.A.M (G.N.S.A.G A.), Italy}\\
\small \msp Dipartimento di Matematica e Informatica \\
\small \msp Universit\`a di Messina \\
\small \msp Messina\\
\small Italia\\
{\small \vspace{1ex} \tt atripodi@unime.it}
 }

\date{ }
\maketitle

\begin{abstract}
Let $K_v$ be the complete graph of order $v$ and $F$ be a set of 1-factors of $K_v$. In this article we
study the existence of a resolvable decomposition of $K_v-F$ into 3-stars when $F$ has the minimum number of 1-factors.
We completely solve the case in which $F$ has the minimum number of 1-factors, with the possible exception of
$v\in\{40,44,52,76,92,100,280,284,\\328,332,428,472,476,572\}$.
\end{abstract}

\vbox{\small
\vspace{5 mm}
 }

\section{Introduction
 }\label{introduzione}

Given a collection of graphs $\cH$, an {\em $\cH$-decomposition} of
a graph $G$(also called $\cH$-{\em design}) is a decomposition of the edges of $G$ into isomorphic
copies of graphs from $\cH$; the copies of $H\in\cH$ in the
decomposition are called {\em blocks}. Such a decomposition is
called {\em resolvable} if it is possible to partition the blocks
into {\em classes} $\cP_i$ such that every point of $G$ appears
exactly once in some block of each $\cP_i$.

A resolvable $\cH$-decomposition of $G$ is sometimes also referred
to as a {\em $\cH$-factorization of $G$}; a class can be called a
{\em $\cH$-factor of $G$}. The case where $\cH$ is a single edge
($K_2$) is known as a {\em $1$-factorization of $G$} and it is well
known to exist for $G=K_v$ if and only if $v$ is even. A single
class of a $1$-factorization, a pairing of all points, is also known
as a {\em $1$-factor\/} or a {\em perfect matching}.

In many cases we wish to impose further constraints on
the classes of an $\cH$-decomposition.
 For example, a class is called {\em uniform\/} if every
block of the class is isomorphic to the same graph from $\cH$. Of
particular note is the result of Rees (\cite{R}) which finds necessary
and sufficient conditions for the existence of uniform $\{K_2,
K_3\}$-decompositions of $K_v$. Uniformly resolvable decompositions
of $K_v$ have also been studied in \cite{DQS}, \cite{GM}, \cite{HR},
 \cite{KMT}, \cite{M}, \cite{MT}, \cite{S1}, \cite{S2}, \cite{SG} and \cite{S3}. Moreover, recently in the case $\cH=\{G_1, G_2\}$
the question of the existence of a uniformly resolvable decomposition of $ K_{v}$ into $r$ classes of
$G_1$ and $s$ classes of $G_2$ have been studied  in the case in which the
number\/ $s$ of\/ $G_2$-factors is maximum. Dinitz, Ling and Danziger
(\cite{DLD}) have solved the case $\cH=\{K_2, K_4\}$ and Kucukcifci, Milici and Tuza (\cite{KMT})
the case  $\cH=\{K_3, K_{1,3}\}$.
In what follows, we will denote by $(a_1;a_2,a_3,a_4)$ the {\em $3$-star\/}, $K_{1,3}$ having
vertex set $\{a_1,a_2,a_3,a_4\}$ and edge set $\{\{a_1,a_2\}, \{a_1,a_3\}, \{a_1,a_4\}\}$.
We will use the notation $(K_2, K_{1,3})$-URD$(v;r,s)$ to denote a uniformly resolvable decomposition of $K_{v}$
into $r$ classes containing only copies of $K_2$ (i.e. 1-factors) and $s$ classes containing only copies of 3-stars.

In this paper, the main purpose is to investigate the existence of a  $(K_2, K_{1,3})$-URD$(v;r,s)$ in the case in which $s>0$ and $r$ is minimum. In particular, we will prove the following result:\\

\noindent \textbf{Main Theorem.} {\em For each\/  $v\equiv
0\pmod{4}$, there exists a\/ $(K_2, K_{1,3})$-URD$(v; $ $r(v),\frac{2(v-1-r(v))}{3})$, with $r(v)$ as in the Table 1 and with the possible exception of  $v\in\{40,44,52,76,92,100,280,284,328,332,428,472,476,572\}$}.

\vspace{4 mm}

 \begin{minipage}[t]{\textwidth}
\begin{center}
\begin{tabular}{|c|c|}
\hline
  $v  $ &  {$r(v)$}
\\
\hline
$0 \pmod{12}$ & $5 $\\
$4 \pmod{12}$ & $3$\\
$8 \pmod{12}$ & $1$\\

\hline
 \end{tabular}

\bigskip

Table 1: The set $r(v)$.

 \end{center}\end{minipage}

\vspace{4 mm}

\section{Necessary conditions}

In this section we will give necessary conditions for the existence
of a uniformly resolvable decomposition of $K_v$ into
$r$  1-factors and $s$ classes of
3-stars, $s>0$.

\begin{lemma}
\label{lemmaP0} If there exists a\/ $(K_2, K_{1,3})$-URD$(v;r,s)$, $s>0$,  then\/ $v\equiv 0\pmod{4}$ and\/ $s\equiv 0\pmod{4}$.
\end{lemma}

\begin{proof}
Assume that there exists a $(K_2, K_{1,3})$-URD$(v;r,s)$ $D$, $s>0$. By resolvability it follows that
$v\equiv 0\pmod{4}$. Counting the edges of $K_v$ that
appear in $D$ we obtain

$$\frac{rv}{2}+\frac{3sv}{4}=\frac{v(v-1)}{2}$$
 and hence
\begin{equation}
  2r+3s=2(v-1). \end{equation}

Denote by $B$ the set of $s$ parallel classes of $3$-stars and
by $R$ the set of $r$ parallel classes of $K_2$. Since the
classes of $R$ are regular of degree $1$, we have that every vertex
$x$ of $K_v$ is incident with $r$ edges in $R$ and
$(v-1)-r$ edges in $B$. Assume that the vertex $x$
appears in $a$ classes with degree $3$ and in $b$ classes with degree 1
in $B$. Since

$$a+b=s \ \ \mbox{and}\ \
3a+b=v-1-r,$$ 
the equality (1) implies that
$$2(v-1-3a-b)+3(a+b)=2(v-1)$$
and hence $$b=3a.$$
This completes the proof.
\end{proof}

\begin{lemma}
\label{lemmaP1} A\/ $(K_2, K_{1,3})$-URD$(v;0,s)$   does not exist for any $v\geq4$.
\end{lemma}
\begin{proof}

Suppose that there exists a uniformly resolvable decomposition $D$
of $K_{v}$ into $s$ classes containing only copies of 3-stars with $s>0$.
Counting the edges of $K_v$ that appear in $D$ we obtain \\
$$s=\frac{2(v-1)}{3}.$$
Since, by Lemma \ref{lemmaP0}, $s=4t$ it follows $$2(v-1)=12t,$$
which is a contradiction, because $v-1$ cannot be even for any $v\geq4$.
\end{proof}

Given $v\equiv 0\pmod{4}$, define $J(v)$ according to the following table:

\vspace{4 mm}

 \begin{minipage}[t]{\textwidth}
\begin{center}
\begin{tabular}{|c|c|}
\hline
  $v  $ &  {$J(v)$}
\\
\hline
$0 \pmod{12}$ & $\{(v-1-6x, 4x), x=0,1,\ldots,\frac{v-6}{6}\} $\\
$4 \pmod{12}$ & $\{(v-1-6x, 4x), x=0,1,\ldots,\frac{v-4}{6}\}$\\
$8 \pmod{12}$ & $\{(v-1-6x, 4x), x=0,1,\ldots,\frac{v-2}{6}\}$\\

\hline
 \end{tabular}

\bigskip

Table 2: The set $J(v)$.

 \end{center}\end{minipage}

\vspace{4 mm}

\begin{lemma}
\label{lemmaP2} If there exists a\/ $(K_2, K_{1,3})$-URD$(v;r,s)$  then\/  $(r,s)\in J(v)$.
\end{lemma}
\begin{proof}
Let $D$ be a $(K_2, K_{1,3})$-URD($v;r,s)$. Lemma \ref{lemmaP0} gives $s\equiv 0\pmod{4}$,
Equation (1) $r\equiv (v-1)\pmod{3}$ \ and \
so
\begin{itemize}
   \item
if $v\equiv 0\pmod{12}$, then  $r \equiv2\pmod{3}$,
  \item
if $v\equiv 4\pmod{12}$,  then$r \equiv0\pmod{3}$,
 \item
if   $v\equiv 8\pmod{12}$, then $r \equiv1\pmod{3}$.
\end{itemize}
Letting $s=4x$ in the Equation (1), we have $r=(v-1)-6x$;
since $r$ and $s$ cannot be negative, and $x$ is an integer, the value of
$x$ has to be in the range as given in the definition of $J(v)$.
\end{proof}

\section{Costructions and related structures}

In this section we will introduce some useful definitions, results and discuss
constructions we will use in proving the main result. For missing
terms or results that are not explicitly explained in the paper,
the reader is referred to \cite{CD} and its online updates.
For some results below, we also cite this handbook instead of the
 original papers.
A (resolvable) $\cH$-decomposition of the complete multipartite
graph with $u$ parts each of size $g$  is known as a (resolvable)
group divisible design $\cH$-(R)GDD of type $g^u$, the parts of size
$g$ are called the groups of the design. When $\cH = K_n$ we will
call it an $n$-(R)GDD.

A $(K_2, K_{1,3})$-URGDD $(r,s)$ of type $g^u$ is a uniformly resolvable
decomposition of the complete multipartite graph with $u$ parts
each of size $g$ into $r$ 1-factors
and $s$ classes containing only copies of 3-stars.

If the blocks of an $\cH$-GDD of type ${g^u}$ can be partitioned into partial
parallel classes, each of which contain all points except those of one group,
we refer to the decomposition as a {\em frame}. When $\cH = K_n$ we will
call it an $n$-{\em frame} and it is easy to deduce that the number of partial parallel classes missing a specified group $G$ is
$\frac{|G|}{n-1}$.

An incomplete resolvable $(K_{2}, K_{1,3})$-decomposition of $K_{v+h}$ with a hole of size $h$ is a
$(K_{2}, K_{1,3})$-decomposition of $K_{v+h}- K_h$ in which there are two types of
classes, {\em partial} classes which cover every point except
those in the hole (the points of $K_h$ are referred to as the {\em hole}) and  {\em full} classes which cover every point of $K_{v+h}$.
Specifically a $(K_{2}, K_{1,3})$-IURD$(v+h, h; [r_1, s_1], [\bar{r}_1, \bar{s}_1])$ is a uniformly
resolvable $(K_{2}, K_{1,3})-$decomposition of $K_{v+h}-K_h$ with $r_1$ 1-factors and $s_1$ classes of 3-stars which cover only the points not in the hole, $\bar{r}_1$ 1-factors and  $\bar{s}_1$ classes of 3-stars which cover every point of $K_{v+h}$.

We now recall the existence of
some $4$-RGDDs and 4-frames we will need in the proof.

\begin{lemma}
\label{lemma R1} {\rm (\cite{CD}, \cite{S3})} There exists a\/ $4$-RGDD of type
\begin{itemize}
\item $4^{t}$  for each\/ $t\equiv 1\pmod{3}$, $t \geq 4$;
\item $3^{t}$  for each\/ $t\equiv 0\pmod{4}$, $t \geq 4$;
\item $2^{22}$, $2^{106}$ and $2^{142}$.

\end{itemize}
\end{lemma}

\begin{lemma}
\label{lemma RF} {\rm (\cite{S3})} There exists a\/ $4$-frame of type
 $6^{t}$  for each\/ $t\equiv 1\pmod{2}$, $t \geq 5$ with the possible exception of $t\in\{7,23,27,35,39,47\}$.
\end{lemma}

We also need the following definitions. Let $(s_1, t_1)$ and $(s_2, t_2)$
be two pairs of non-negative integers. Define $(s_1, t_1) +(s_2,
t_2)=(s_1+s_2, t_1+t_2)$. If $X$ and $Y$ are two sets of pairs of
non-negative integers, then $X+Y$ denotes the set $\{(s_1, t_1)
+(s_2, t_2) : (s_1, t_1)\in X, (s_2, t_2) \in Y \}$. If $X$ is a set
of pairs of non-negative integers and $h$ is a positive integer,
then  $h * X$ denotes the set of all pairs of non-negative integers
which can be obtained by adding any $h$ elements of $X$ together
(repetitions of elements of $X$ are allowed).

\begin{thm}
\label{thmR2} Let\/ $v$, $g$, $t$  and\/ $u$   be non-negative
integers such that\/ $v=gtu$. If there exists
\begin{itemize}
\item
[$(1)$] a\/ $4$-RGDD of type\/ $g^{u}$;

\item
[$(2)$] a $(K_2, K_{1,3})$-URGDD$(r_1,s_1)$ of\/
type $t^{4}$ with\/ $(r_1, s_1)\in J_1$;

\item
[$(3)$] a $(K_2, K_{1,3})$-URD$(gt;r_2,s_2)$,\/ with\/ $(r_2, s_2)\in J_2$;
\end{itemize}
then there exists a\/ $(K_2, K_{1,3})$-URD$(v;r,s)$  for
each\/ $(r,s)\in J_2+ h\ast J_1$, where\/ $h=\frac{g(u-1)}{3}$ is the
number of parallel classes of the\/ $4$-RGDD of type\/ $g^{u}$.
\end{thm}
\begin{proof}
Let $\cG$ be a $4$-RGDD of type $g^{u}$,
with $u$ groups $G_i$, $i=1,2,\ldots ,u$, of size $g$; let $R_1,R_2,\ldots,R_{h}$, $h=\frac{g(u-1)}{3}$, be the
parallel classes of this $4$-RGDD. Expand each point $t$ times and for each block $b$ of a given resolution class
of $\cG$ place on $b\times\{1,2,\ldots,t\}$ a copy of a $(K_2, K_{1,3})$-URGDD$(r_1,s_1)$  of\/
type $t^{4}$ with $(r_1, s_1)\in J_1$. For each
$i=1,2,\ldots,u$, place on $G_i\times\{1,2,\ldots,t\}$ a copy of a
$(K_2, K_{1,3})$-URD$(gt;r_2,s_2)$
with\/ $(r_2, s_2)\in J_2$. The result is a $(K_2, K_{1,3})$-URD$(v;r,s)$ with $(r,s)\in \{J_2+ h\ast
J_1$\}.
\end{proof}

\begin{thm}
\label{thmR3} Let $v$, $g$, $t$, $h$ and $u$  be non-negative integers
such that $v=gtu+h$. If there exists
\begin{itemize}
\item
[$(1)$] a $4$-frame $\cF$ of type $g^{u}$;

\item
[$(2)$] a $(K_{2}, K_{1,3})$-URD$(h;r_1,s_1)$  with $(r_1,
s_1)\in J_1$;

\item
[$(3)$] a $(K_{2}, K_{1,3})$-URGDD$(r_2,s_2)$ of type $t^{4}$ with $(r_2,
s_2)\in J_2$;

\item
[$(4)$] a $(K_{2}, K_{1,3})$-IURD$(gt+h,h; [r_1, s_1], [r_3,s_3])$  with
$(r_1, s_1)\in J_1 $ and $(r_3,s_3)\in J_3$= $ \frac{g}{3}\ast J_2$;
\end{itemize}
then exists a $(K_{2}, K_{1,3})$-URD$(v+h;r,s)$  for each
$(r,s)\in J_1+ u\ast J_3$.
\end{thm}

\begin{proof}

Let $\cF$ be a $4$-frame of type $g^{u}$ with groups $G_i$, $i=1,2,\ldots,u$; expand each point $t$ times and add a set $H=\{a_1,a_2,\ldots,a_h\}$.
For $j=1,2,\ldots,\frac{g}{3}$,  let $p_{i,j}$ be the $j$-th partial parallel class which miss the group $G_i$;  for each $b\in p_{i,j}$, place on $b\times\{1,2,\ldots,t\}$ a copy $D_{i,j}^b$  of a $(K_{2}, K_{1,3})$-URGDD$(r_2,s_2)$ of type $t^4$,  with $(r_2,s_2)\in J_2$;   place on $G_i\times\{1,2,\ldots,t\}\cup H$ a copy $D_i$ of a $(K_{2}, K_{1,3})$-IURD$(gt+h,h;[r_1, s_1], [r_3,s_3])$   with $H$ as hole, $(r_1,s_1)\in J_1$ and $(r_3,s_3)\in J_3$= $ \frac{g}{3}\ast J_2$. Now  combine all together the  parallel classes of $D_{i,j}^b$, $b\in p_{i,j}$,
 along with the full classes of   $D_i$ so to obtain $r_3$ 1-factors and $s_3$ classes
of 3-stars, $(r_3,s_3)\in J_3$, on $\cup_{i=1}^{u} G_i\times\{1,2,\ldots,t\}\cup H $.
Fill the hole $H$ with a copy $D$ of
$(K_{2}, K_{1,3})$-URD$(h;r_1,s_1)$  with $(r_1,s_1)\in J_1$ and
combine the classes of $D$ with the partial classes of $D_i$  so to obtain
$r_1$ 1-factors and $s_1$ classes of 3-stars on $\cup_{i=1}^{u} G_i\times\{1,2,\ldots,t\}\cup H$.
The result is a $(K_{2}, K_{1,3})$-URD$(v+h;r,s)$  for each $(r,s)\in J_1+ u\ast J_3$.

\end{proof}

\section{Small cases}
\begin{lemma}
\label{lemmaD1} There exists a \/ $(K_2, K_{1,3})$-URGDD$(0,4)$ of type\/ $2^{4}$.
\end{lemma}

\begin{proof}
Take the groups to be $\{0,1\}, \{2,3\}, \{4,5\}, \{6,7\}$ and the
classes as listed below:\\ $\{(0;2,4,6),(1;3,5,7)\}$, $\{
(2;4,1,6),(3;5,0,7)\}$, $\{(5;2,0,7),(4;1,3,6)\}$,\\
$\{(6;1,3,5),(7;0,4,2)\}$.
\end{proof}

\begin{lemma}
\label{lemmaD2} There exists a \/ $(K_2, K_{1,3})$-URD$(8;1,4)$.
\end{lemma}
\begin{proof}
The assertion follows by Lemma \ref{lemmaD1}.
\end{proof}

\begin{lemma}
\label{lemmaD3} There exists a \/ $(K_2, K_{1,3})$-URD$(12;5,4)$.
\end{lemma}

\begin{proof}
Let $V(K_{12}$)=$\mathbb{Z}_{12}$, and the classes as listed below:\\
$\{(0;4,5,6)$, $(7;8,9,10)$, $(11;1,2,3)$\}$, $\{ $(1;5,6,7)$, $(4;9,10,11)$, $(8;0,2,3)\}$,\\
$\{(2;4,6,7)$, $(5;8,10,11)$, $(9;0,1,{{3}})$\}$, $\{$(3;4,5,7)$, $(6;8,9,11)$, $(10;0,1,2)\}$,\\
$\{\{0,7\}, \{1,4\}, \{2,5\}, \{3,6\}, \{8,11\}, \{9,10\}\}$,\\ $\{\{0,1\}, \{3,10\}, \{2,9\}, \{4,8\}, \{5,6\}, \{7,11\}\}$,\\
$\{\{0,11\}, \{1,8\}, \{2,3\}, \{4,7\}, \{6,10\}, \{5,9\}\}$,\\ $\{\{0,2\}, \{1,3\}, \{4,5\}, \{6,7\}, \{8,10\}, \{9,11\}\}$,\\
$\{\{0,3\}, \{1,2\}, \{5,7\}, \{4,6\}, \{8,9\}, \{10,11\}\}$.
\end{proof}

\begin{lemma}
\label{lemmaD4} There exists a \/ $(K_2, K_{1,3})$-URGDD$(4,8)$ of type\/
$8^{3}$.
\end{lemma}

\begin{proof}
Let $\{a_0,a_1\ldots, a_7\}$, $\{b_0,b_1,\ldots, b_7\}$ and $\{c_0,c_1,\ldots,c_7\}$ be the groups and the classes as listed below:\\
$\{(a_0;b_1,b_2,b_3)$, $(b_0;c_0,c_2,c_6)$, $(c_4;a_1,a_2,a_3)$, $(b_7;c_1,c_5,c_7)$, $(c_3;a_4,a_5,a_6)$,\\$(a_7;b_4,b_5,b_6)\}$,\\
$\{(a_1;b_0,b_2,b_3)$, $(b_1;c_1,c_3,c_7)$, $(c_5;a_0,a_2,a_3)$,
$(b_4;c_2,c_4,c_6)$, $(c_0;a_7,a_5,a_6)$,\\ $(a_4;b_7,b_5,b_6)\}$,\\
$\{(a_2;b_1,b_0,b_3)$, $(b_2;c_0,c_2,c_4)$, $(c_6;a_1,a_0,a_3)$,$(b_5;c_3,c_5,c_7)$, $(c_1;a_4,a_7,a_6)$,\\$(a_5;b_4,b_7,b_6)\}$,\\
$\{(a_3;b_1,b_2,b_0)$, $(b_3;c_1,c_3,c_5)$, $(c_7;a_1,a_2,a_0)$,
$(b_6;c_0,c_4,c_6)$, $(c_2;a_4,a_5,a_7)$, \\$(a_6;b_4,b_5,b_7) \}$,\\
$\{(a_0;b_4,b_5,c_4)$, $(b_7;c_0,{{c_6}},a_1)$, $(c_2;b_6,a_2,a_3)$,
$(b_0;c_3,c_5,a_6)$, $(c_7;a_4,a_7,b_3)$, \\$(a_5;b_1,b_2,c_1)\}$,\\
$\{(a_1;b_5,b_6,c_5)$, $(b_4;c_1,a_2,c_7)$, $(c_3;a_3,a_0,b_7)$,
$(b_1;a_7,c_0,c_6)$, $(c_4;a_4,a_5,b_0)$, \\$(a_6;b_2,b_3,c_2) \}$,\\
$\{(a_2;b_6,b_7,c_6)$, $(b_5;a_3,c_2,c_4)$, $(c_0;a_1,a_0,b_4)$,
$(b_2;a_4,c_1,c_7)$, $(c_5;b_1,a_5,a_6)$,\\$(a_7;b_0,b_3,c_3)\}$,\\
$\{(a_3;b_4,b_7,c_7)$, $(b_6;a_0,c_3,c_5)$, $(c_1;a_1,a_2,b_5)$,
$(b_3;c_2,c_4,a_5)$, $(c_6;b_2,a_6,a_7)$,\\ $(a_4;c_0,b_0,b_1)\}$,\\
$\{\{a_0, b_0\}, \{a_1, b_1\}, \{a_2, b_2\}, \{a_3, b_3\}, \{a_4, c_5\},  \{a_5, c_6\},  \{a_6, c_7\},  \{a_7, c_4\},  \{b_4, c_3\},\\
\{b_5, c_0\},  \{b_6, c_1\},  \{b_7, c_2\}\}$,\\
$\{\{a_0, c_1\}, \{a_1, c_2\}, \{a_2, c_3\}, \{a_3, c_0\}, \{a_4, b_3\},  \{a_5, b_0\},  \{a_6, b_1\},  \{a_7, b_2\},  \{b_4, c_5\},\\
\{b_5, c_6\},  \{b_6, c_7\},  \{b_7, c_4\}\}$,\\
$\{{{\{a_0, c_2\}, \{a_1, c_3\}, \{a_2, c_0\}, \{a_3, c_1\}}}, \{a_4, b_4\},  \{a_5, b_5\},  \{a_6, b_6\},  \{a_7, b_7\},  \{b_0, c_7\},\\
\{b_1, c_4\},  \{b_2, c_5\},  \{b_3, c_6\}\}$,\\
$\{\{a_0, b_7\}, \{a_1, b_4\}, \{a_2, b_5\}, \{a_3, b_6\}, \{a_4, c_6\},  \{a_5, c_7\},  \{a_6, c_4\},  \{a_7, c_5\},  \{b_0, c_1\},\\
\{b_1, c_2\},  \{b_2, c_3\},  \{b_3, c_0\}\}$.

\end{proof}

\begin{lemma}
\label{lemmaD5} There exists a \/ $(K_2, K_{1,3})$-URD$(24;5,12)$.

\end{lemma}

\begin{proof}
Take a $({{K_2,K_{1,3}}})$-URGDD$(4,8)$  of type $8^{3}$, which exists by Lemma {{\ref{lemmaD4}}}. Place on each of the groups a copy of a $(K_2, K_{1,3})$-URD$(8; 1,4)$ which exists by
 Lemma {{\ref{lemmaD2}}}. This completes the proof.
\end{proof}

\begin{lemma}
\label{lemmaD6}  There exists \/ $(K_2, K_{1,3})$-URGDD$(0,8)$ of type\/
$4^{4}$.
\end{lemma}

\begin{proof}
Take the groups to be $\{x_1,x_2,x_3, x_4\}$, $\{a_1,a_2,a_3, a_4\}$, $\{b_1,b_2,b_3,, b_4\}$ and $\{c_1,c_2,c_3,
c_4\}$ and the classes are obtained by reducing subscripts modulo 4 the following base blocks:\\
$\{(a_1;b_2, c_3, x_2)\}$, $\{(b_1;a_3, c_3, x_3)\}$, $\{(c_1;b_2, a_2, x_3)\}$, $\{(x_1;b_2, a_3, c_2)\}$,
$\{(a_1;b_1, c_1, $ $x_1)\}$, $\{(b_1;a_2, c_1, x_1)\}$, $\{(c_1;b_4, a_4, x_1)\}$, $\{(x_1;b_4, a_2, c_4)\}$.

\end{proof}

\begin{lemma}
\label{lemmaD7} There exists a \/ $(K_2, K_{1,3})$-IURD$(16,4;[3,0],[0,8])$.

\end{lemma}

\begin{proof}
Start from the $(K_2, K_{1,3})$-URGDD$(0,8)$  of type $4^{4}$ of Lemma
\ref{lemmaD6} and fill in the  groups $\{a_1, a_2, a_3, a_4\}$, $\{b_1, b_2, b_3, b_4\}$ and $\{c_1, c_2, c_3, c_4\}$ with  a copy of a  $(K_2, K_{1,3})$-URD$(4;3,0)$ to obtain
a  \/ $(K_2, K_{1,3})$-IURD$(16,4;[3,0],[0,8])$ with $\{x_1,x_2,x_3, x_4\}$ as hole.
\end{proof}

\begin{lemma}
\label{lemmaD8} There exists a \/ $(K_2, K_{1,3})$-URD$(16;3,8)$.
\end{lemma}
\begin{proof}
The assertion follows by Lemma \ref{lemmaD7}.
\end{proof}\

\begin{lemma}
\label{lemmaD9}  There exists a \/ $(K_2, K_{1,3})$-IURD$(28,4;[3,0],[0,16])$.

\end{lemma}

\begin{proof}

Let the point set be $V=\{x, a, b, c, d, f, g\}\times\{1, 2, 3, 4\}$ and let\\ $\{x_1, x_2, x_3, x_4\}$ be the hole.

\begin{itemize}
\item
Take  16 classes of 3-stars on $V$ listed bellow:\\
$\{(x_i; a_{i+3}, b_{i+3},d_{i+3}), (a_i; a_{i+1}, c_{i+3},g_{i+2})$, $(b_i; b_{i+{{1}}}, c_{i+2},f_{i+3}), (d_i; d_{i+1}, g_{i+3},$ $f_{i+2})$,
 $(c_i; c_{i+1}, a_{i+2},x_{i+3}), (f_i; f_{i+1}, b_{i+2},x_{i+1})$,$ (g_i; g_{i+1}, x_{i+2},d_{i+2}), i\in Z_4\}$,\\
$\{(x_i; c_{i+2}, f_{i+1},g_{i+1}), (a_{i+3}; {{c}}_{i}, x_{i+3},b_{i})$, $(b_{i+2}; x_{i+2}, d_{i+2},f_{i+2}), (d_{i+1}; x_{i+1},b_{i+3},$ $ a_{i+2})$,
 $(c_{i+1}; a_{i+1}, g_{i+2},d_{i+3}), (f_i; {{c}}_{i+3}, a_{i},g_{i+3})$, $(g_i; f_{i+3}, b_{i+1},d_{i}), i\in Z_4\}$,\\
$\{(x_i; a_{i+{{1}}}, b_{i+1}, d_{i+1}), (f_{i+3}; a_{i+2}, x_{i+1},c_{i})$, $(c_{i+2}; x_{i+3}, d_{i+3},f_{i}), (g_{i+1}; x_{i+2}, d_{i},a_{i})$,\\
 $(a_{i+3}; g_{i+2}, b_{i+2}, g_{i+3}), (b_i; g_{i}, c_{i+1}, f_{i+1})$, $( d_{i+2}; {{b}}_{i+3}, f_{i+2}, c_{i+3}), i\in Z_4\}$,\\
$\{(x_i; f_{i}, c_{i}, g_{i}), (a_{i+1}; d_{i+1}, x_{i+3},b_{i+3})$, $(b_{i}; x_{i+2}, c_{i+3},a_{i}), (d_{i+3}; x_{i+1}, f_{i+2},a_{i+2})$,\\
 $(g_{i+3}; b_{i+1}, b_{i+2}, f_{i+3}), (f_{i+1}; a_{i+3}, d_{i}, c_{i+1})$, $( c_{i+2}; d_{i+{{2}}}, g_{i+1}, g_{i+2}), i\in Z_4\}$.

 \item
Take  three 1-factors on $\{\{a, b, c, d, f, g\}\times\{1, 2, 3, 4\}\}$:\\
$\{\{a_i, f_{i+3}\}, \{(b_i, d_{i+1}\}, \{c_i, g_{i+2}\}, i\in Z_4\}$, $\{\{a_i, d_{i+2}\}, \{b_i, c_{i}\}, \{f_i, g_{i+2}\}, i\in Z_4\}$,\\
$\{\{a_1, a_3\}, \{a_2, a_4\}, \{b_1, b_3\}$, $\{b_2, b_4\}, \{c_1, c_3\}, \{c_2, c_4\}$, $\{d_1, d_3\}, \{d_2, d_4\},
\{f_1, f_3\}$,\\ $\{f_2, f_4\}, \{g_1, g_3\}, \{g_2, g_4\}\}$.

\end{itemize}
\end{proof}

\begin{lemma}
\label{lemmaD10} There exists a \/ $(K_2, K_{1,3})$-URD$(28;3,16)$.
\end{lemma}
\begin{proof}
The assertion follows by Lemma {{\ref{lemmaD9}}}.
\end{proof}

\begin{lemma}
\label{lemmaD11}  There exists a \/ $(K_2, K_{1,3})$-IURD$(20,8;[1,4],[0,8])$.
\end{lemma}

\begin{proof}
Let the point set be $V=Z_{20}$ and let $H=\{0,1,\ldots,7\}$ be the hole.
\begin{itemize}
 \item Take the 8 classes of 3-stars on $V$ listed bellow:\\
$\{(0; 8, 9, 10),  (1; 11, 12, 13),   (2; 14, 15, 16), (17; 3, 4, 5),  (18; 6, 7, 19)\}$,\\
$\{(0; 11, 12, 13), ( 1; 8, 9, 10), ( 2; 17, 18, 19), ( 14; 3, 4, 5), ( 15; 6, 7, 16)\}$,\\
$\{ (3; 8, 9, 10), (  4; 11, 12, 15 ), ( 5; 16, 18, 19), ( 13; 2, 6, 7), ( 14; 0 ,1, 17)\}$,\\
$\{ (3; 11, 12, 13), (  4; 8, 9, 16 ), (  6; 14, 17, 19), (  10; 2, 5, 7  ), ( 15; 0, 1, 18)\}$,\\
$\{(5; 8, 9, 11), (  7; 14, 16, 17 ), ( 10; 4, 6, 19 ), (  12; 2, 13, 15), (  18; 0, 1, 3)\}$,\\
$\{(6; 8, 9, 16  ), (  11; 2, 7, 10  ), ( 15; 3, 5, 13 ), (  17; 0, 1, 18  ), (  19; 4, 12, 14)\}$,\\
$\{ (7; 8, 12, 19  ), (  9; 2, 10, 14), ( 11; 6, 15, 18 ), ( 13; 4, 5, 17  ), (   16; 0, 1, 3)\}$,\\
$\{(8; 2, 10, 15 ), (  9; 7, 11, 13 ), ( 12; 5, 6 ,17), ( 18; 4, 14, 16 ), ( 19; 0, 1, 3)\}$.
\item Take the 4  partial  classes of 3-stars  on $V$ listed bellow:\\
$\{(8; 9, 12, 18    ), (    11; 13, 14, 19   ), (   17; 10, 15, 16)\}$,\\
$\{( 8; 11, 13, 17  ), (     10; 12, 15, 18  ), (    16; 9 ,14, 19)\}$,\\
$\{(9; 17, 18, 19  ), (     14 ;10, 12, 15 ), (     16 ;8, 11, 13)\}$,\\
$\{(12; 9, 11, 16   ), (    13 ;10, 14, 18 ), (     19 ;8 ,15, 17)\}$.
\item  Take the partial 1-factors on $V$:\\
$\{8,14\},\{9,15\},\{10,16\},\{11,17\},\{12,18\},\{13,19\}$.
\end{itemize}
\end{proof}

\begin{lemma}
\label{lemmaD12} There exists a \/ $(K_2, K_{1,3})$-URD$(20;1,{{12}})$.
\end{lemma}
\begin{proof}
The assertion follows by Lemmas \ref{lemmaD11}  and \ref{lemmaD2}.
\end{proof}

\begin{lemma}
\label{lemmaD13} There exists a \/ $(K_2, K_{1,3})$-URD$(v;3, \frac{2(v-4)}{3})$
for $v=88,424,568$.
\end{lemma}

\begin{proof}
Start from a 4-RGDD $\cG$ of type $2^{\frac{v}{4}}$ which exists for $v=88,424,568$ \cite{S3}. Give weight $2$ to
every point of $\cG$ and for each block of a given
resolution class of $\cG$ place a copy of a $(K_2,K_{1,3})$-RGDD$(0,4)$ of type
$2^{4}$ which exists by Lemma \ref{lemmaD1}. Fill each group of size 4 with a copy of a
$(K_2, K_{1,3})$-URD($4;3,0)$. Applying Theorem \ref{thmR2} with
$g=t=2$ and $u=\frac{v}{4}$, we obtain a \/ $(K_2, K_{1,3})$-URD$(v;3,\frac{2(v-4)}{3})$.
\end{proof}

\section{Main results}

\begin{lemma}
\label{lemmaC1} For every\/ $v\equiv 0\pmod{24}$ there exists a\/
$(K_2, K_{1,3})$-URD $(v;5,$ $\frac{2(v-6)}{3})$.

\end{lemma}

\begin{proof}
Let $v=24s$. The case $s=1$ corresponds to a $(K_2, K_{1,3})$-URD $(24;5,12)$ which exists by Lemma
\ref{lemmaD5}. For $s>1$, start from a $4$-RGDD of type $3^{4s}$  (\cite{CD}) and apply Theorem \ref{thmR2} with $t=2$  to obtain a $(K_2, K_{1,3})$-URD $(v;5, \frac{2(v-6)}{3})$
(the input designs are a $(K_2, K_{1,3})$-URGDD$(0,4)$ of type $2^{4}$ by Lemma \ref{lemmaD1}, and a $(K_2, K_{1,3})$-URD($6;5,0)$).
\end{proof}

\begin{lemma}
\label{lemmaC2} For every\/ $v\equiv 12\pmod{24}$ there exists a\/
$(K_2, K_{1,3})$-URD $(v;5,$ $ \frac{2(v-6)}{3})$.

\end{lemma}

\begin{proof}
Let $v=12(2s+1)$, $s\geq0$. The case $s=0$ corresponds to a
$(K_2, K_{1,3})$-URD $(12;5,4)$ which exists by Lemma \ref{lemmaD3}.
For $s\geq 1$ start from a $(K_2, K_{1,3})$-URGDD $(0, \frac{2(v-12)}{3})$ of type $12^{1+2s}$,
which exists by Lemma 5.4 of  \cite{KMT}. Filling each group of size
12 with a copy of a $(K_2, K_{1,3})$-URD($12;5,4)$, which exists by
Lemma \ref{lemmaD3}, gives a $(K_2, K_{1,3})$-URD $(v;5, \frac{2(v-6)}{3})$.
\end{proof}

\begin{lemma}
\label{lemmaC3} For every\/ $v\equiv 8\pmod{24}$ there exists a\/
$(K_2, K_{1,3})$-URD $(v;1, $ $\frac{2(v-2)}{3})$.

\end{lemma}

\begin{proof}
Let $v=8+24s$. The case $s=0$ corresponds to a $(K_2, K_{1,3})$-URD $(8;1,4)$ which exists by Lemma \ref{lemmaD2}. For $s\geq1$, start from a $4$-RGDD  of type $4^{1+3s}$  (\cite{CD}) and apply Theorem \ref{thmR2} with $t=2$
to obtain a $(K_2, K_{1,3})$-URD $(v;1, \frac{2(v-2)}{3})$
(the input designs are a $(K_2, K_{1,3})$-URGDD$(0,4)$ of type $2^{4}$ from Lemma \ref{lemmaD1} and a $(K_2, K_{1,3})$-URD($8;1,4)$ from
Lemma \ref{lemmaD2}).
\end{proof}

\begin{lemma}
\label{lemmaC4} For every\/ $v\equiv 16\pmod{24}$, with the possible exception of $v\in\{40,280,328,472\}$, there exists a\/
$(K_2, K_{1,3})$-URD $(v;3, \frac{2(v-4)}{3})$.
\end{lemma}

\begin{proof}
Let $v=16+24s$. The cases $v=16,88,424,568$ are covered by Lemmas \ref{lemmaD8} and \ref{lemmaD13}.  For $v>40$ and $v\neq 280,328,472$  start from a $4$-frame of type $6^{1+2s}$ (\cite{SG}) and apply Theorem \ref{thmR3} with $t=2$ and $h=4$ to obtain a $(K_{2}, K_{1,3})$-URD $(v;3, \frac{2(v-4)}{3} )$ (the input designs are a $(K_{2}, K_{1,3})$-URD $(4;3, 0 )$, a $(K_2, K_{1,3})$-URGDD$(0,4)$ of type $2^{4}$ from Lemma \ref{lemmaD1} and a $(K_{2},
K_{1,3})$-IURD $(16,4;[3,0],[0,8])$  from
Lemma \ref{lemmaD7}).
\end{proof}

\begin{lemma}
\label{lemmaC5} For every\/ $v\equiv 4\pmod{24}$, with the possible exception of $v\in\{52,76,100\}$, there exists a\/
$(K_2, K_{1,3})$-URD $(v;3, \frac{2(v-4)}{3})$.
\end{lemma}

\begin{proof}
Let $v=4+24s$. The case $v=28$ follows by Lemma \ref{lemmaD10}. For $v>100$  start from a $4$-frame of type $12^{s}$ (\cite{S3}) and apply Theorem \ref{thmR3} with $t=2$ and $h=4$
to obtain  a $(K_{2}, K_{1,3})$-URD $(v;3, \frac{2(v-4)}{3} )$
(the input designs are a $(K_{2}, K_{1,3})$-URD $(4;3, 0 )$, a $(K_2, K_{1,3})$-URGDD$(0,4)$ of type $2^{4}$ from Lemma \ref{lemmaD1} and a $(K_{2},
K_{1,3})$-IURD$(28,4;[3,0],[0,16])$   from Lemma \ref{lemmaD9}).
\end{proof}

\begin{lemma}
\label{lemmaC6} For every\/ $v\equiv 20\pmod{24}$, with the possible exception of $v\in\{44,92,284,332,428,476,572\}$, there exists a\/
$(K_2, K_{1,3})$-URD $(v;1, \frac{2(v-2)}{3})$.
\end{lemma}

\begin{proof}
Let $v=20+24s$. The case $v=20$ follows by Lemma \ref{lemmaD12}. For $v>44$ and $v\neq 92, 284, 332, 428, 476,572$  start from a $4-$frame of type $6^{1+2s}$ (\cite{S3}) and apply Theorem \ref{thmR3} with  $t=2$ and
$h=8$ to obtain a $(K_{2}, K_{1,3})$-URD $(v;1, \frac{2(v-2)}{3} )$ (the input designs are a $(K_{2}, K_{1,3})$-URD $(8;1, 4 )$ from Lemma \ref{lemmaD2}, a $(K_2, K_{1,3})$-URGDD$(0,4)$ of type $2^{4}$ from Lemma \ref{lemmaD1} and a $(K_{2}, K_{1,3})$-IURD$(20, $ $8;[1,4],[0,8])$   from Lemma \ref{lemmaD11}).
\end{proof}

Combining Lemmas \ref{lemmaC1}, \ref{lemmaC2}, \ref{lemmaC3}, \ref{lemmaC4}, \ref{lemmaC5} and \ref{lemmaC6}  we obtain the main
theorem of this article.

\begin{thm}
For each\/  $v\equiv
0\pmod{4}$, there exists a\/ $(K_2, K_{1,3})$-URD$(v; r(v), $ $\frac{2(v-1-r(v))}{3})$, with $r(v)$ as in the Table 1 and with the possible exception of  $v\in\{40,44,52,76,92,100,280,284,328,332,428,472,476,572\}$.

\end{thm}

\end{document}